\documentclass[10 pt, pdftex]{article}
\usepackage[margin=2.5cm]{geometry}
\usepackage{longtable}
\usepackage{booktabs}
\usepackage{colortbl}
\usepackage{float}
\usepackage{hyperref}

\usepackage{amsmath,amsthm,amssymb,amscd,graphicx,color,enumerate}
\usepackage[all]{xy}

\newlength{\matrixheight}
\newlength{\padabove}

\theoremstyle{plain}

\newtheorem{thm}{Theorem}[]

\newtheorem{lem}[thm]{Lemma}
\newtheorem{cor}[thm]{Corollary}

\theoremstyle{definition}
\newtheorem{dfn}[thm]{Definition}
\newtheorem{term}[thm]{Terminology}
\newtheorem{warn}[thm]{Warning}
\newtheorem{nota}[thm]{Notation}
\newtheorem{setup}[thm]{Setup}

\newtheorem{exa}[thm]{Example}
\newtheorem{rem}[thm]{Remark}

\theoremstyle{remark}

\DeclareMathOperator{\SL}{SL}
\DeclareMathOperator{\PGL}{PGL}

\DeclareMathOperator{\Cr}{Cr}
\DeclareMathOperator{\NKLT}{NKLT}
\DeclareMathOperator{\nklt}{nklt}
\DeclareMathOperator{\mult}{mult}
\DeclareMathOperator{\divisor}{div}
\DeclareMathOperator{\centre}{z}

\DeclareMathOperator{\Proj}{Proj}

\DeclareMathOperator{\Div}{Div}
\DeclareMathOperator{\NEbar}{\overline{NE}}

\DeclareMathOperator{\Effbar}{\overline{Eff}}

\DeclareMathOperator{\Supp}{Supp}

\newcommand{\oo}{\mathcal{O}}
\newcommand{\QQ}{\mathbb{Q}}
\newcommand{\RR}{\mathbb{R}}
\newcommand{\PP}{\mathbb{P}}
\newcommand{\ZZ}{\mathbb{Z}}

\newcommand{\CC}{\mathbb{C}}
\newcommand{\GG}{\mathbb{G}}

\newcommand{\vf}{\varphi}
\newcommand{\ve}{\varepsilon}
\newcommand{\Kc}[2]{K_{#1}+#2_{#1}}

\begin{document}
\bibliographystyle{plain}

\title{The Sarkisov program for Mori fibred Calabi--Yau pairs}

\author{Alessio Corti \thanks{a.corti@imperial.ac.uk}\\
 Department of Mathematics\\
 Imperial College London\\
 180 Queen's Gate\\
 London, SW7 2AZ, UK\\
\and
 Anne-Sophie Kaloghiros \thanks{anne-sophie.kaloghiros@brunel.ac.uk}\\
 Department of Mathematics\\
 Brunel University\\
 Uxbridge\\
 Middlesex, UB8 3PH,UK}

\date{\today}

\maketitle

\begin{abstract} We prove a version of the Sarkisov program for volume
  preserving birational maps of Mori fibred Calabi--Yau pairs valid in
  all dimensions. Our theorem generalises the theorem of Usnich and
  Blanc on factorisations of birational maps of $(\CC^\times)^2$ that
  preserve the volume form $\frac{dx}{x}\wedge \frac{dy}{y}$.
\end{abstract}

\tableofcontents{}

\section{Introduction}
\label{sec:introduction}

Usnich~\cite{usnich06:_sympl_cp_thomp_t} and Blanc~\cite{MR3080816}
proved that the group of birational automorphisms of $\GG_m^2$ that
preserve the volume form $\frac{dx}{x}\wedge \frac{dy}{y}$ is
generated by $\GG_m^2$, $\SL_2(\ZZ)$, and the birational map:
\[
P\colon (x, y) \dasharrow \big(y, \frac{1+y}{x}\big)
\]
In this paper we prove a generalisation of this result valid in all
dimensions. Our theorem generalises the theorem of Usnich and Blanc in
the same way that the Sarkisov program~\cite{MR1311348},
\cite{MR3019454} generalises the theorem of Noether and Castelnuovo
stating that $\Cr_2$ is generated by $\PGL_3(\CC)$ and a standard
quadratic transformation
\[
C\colon (x_0: x_1: x_2) \dasharrow \Bigl( \frac1{x_0}:\frac1{x_1}:
\frac1{x_2} \Bigr)
\] 

Our main result is the following:

\begin{thm}
  \label{thm:1}
  A volume preserving birational map between Mori fibred Calabi--Yau pairs is a
  composition of volume preserving Sarkisov links.
\end{thm}

In the rest of the introduction, we introduce the terminology needed
to make sense of the statement and, along the way, we state the more
general factorisation theorem~\ref{thm:2} for volume preserving
birational maps of general Calabi--Yau pairs. Theorem~\ref{thm:2} is
used in the proof of the main result and is of independent
interest. We conclude with some additional remarks.

\begin{dfn}
  \label{dfn:1}
  \begin{enumerate}[(1)]
  \item A \emph{Calabi--Yau (CY) pair} is a pair $(X, D)$ of a normal variety
    $X$ and a reduced $\ZZ$-Weil divisor $D\subset X$ such that $K_X+D\sim 0$
    is a Cartier divisor linearly equivalent to $0$.
  \item We say that a pair $(X,D)$ has \emph{(t,~dlt)}, resp.\ \emph{(t,~lc)}
    singularities or that it ``is'' (t,~dlt), resp.\ (t,~lc) if $X$
    has terminal singularities and the pair $(X,D)$ has dlt, resp.\ lc
    singularities. 

    Similarly $(X,D)$ has \emph{(c,~dlt)}, resp.\ \emph{(c,~lc)}
    singularities or ``is'' (c,~dlt), resp.\ (c,~lc) if $X$
    has canonical singularities and the pair $(X,D)$ has dlt, resp.\ lc
    singularities.

  \item We say that a pair $(X, D)$ is \emph{$\QQ$-factorial} if $X$
    is $\QQ$-factorial.
  \end{enumerate}
\end{dfn}

\begin{rem}
  \label{rem:1}
  \begin{enumerate}[(1)]
  \item We use the following observation throughout:
    If $(X,D)$ is a CY pair, then, because $K_X+D$ is an integral
    Cartier divisor, for all geometric valuations $E$,
    $a(E,K_X+D)\in \ZZ$. If in addition $(X,D)$ is lc or dlt, then
    $a(E,K_X+D)\leq 0$ implies $a(E,K+D)=-1$ or $0$.
  \item If $(X,D)$ is a dlt CY pair, then automatically it is (c,~dlt). More
    precisely if $E$ is a geometric valuation with small centre on $X$
    and if the centre $\centre_X E\in \Supp D$, then $a(E,K_X)>0$.

  Indeed, consider a valuation $E$ with small centre on $X$. Then
\[a(E,K_X)=a(E,K_X+D)+\mult_E \overline{D}
\]
therefore $a(E,K_X)\leq 0$ implies $a(E, K_X+D)\leq 0$ and then,
because $K_X+D$ is a Cartier divisor, \emph{either} $a(E, K_X+D)=-1$,
which is impossible because by definition of dlt, see remark~\ref{rem:NKLT}
below, $z=\centre_X E\in X$ is smooth, \emph{or} $a(E, K_X+D)=0$ and
$\mult_E \overline {D} =0$.
  \end{enumerate}
\end{rem}

\begin{dfn}
  \label{dfn:2}
  Let $(X,D_X)$ and $(Y,D_Y)$ be CY pairs. A birational map $\vf\colon
  X \dasharrow Y$ is \emph{volume preserving} if there exists a common
  log resolution
\[
\xymatrix{ & W\ar[dl]_p\ar[dr]^q & \\
X \ar@{-->}[rr]^\varphi& & Y
}
\]
such that $p^\star (K_X+D_X)= q^\star (K_Y+D_Y)$. It is essential to
the definition that we are requiring equality here not linear
equivalence.
\end{dfn}

\begin{rem}
  \label{rem:7}
  If $(X,D_X)$ is a CY pair then there is a (unique up to
  multiplication by a nonzero constant) rational differential
  $\omega_X\in \Omega^n_{k(X)/k}$ (where $n=\dim X$) such that
  $D_X+\divisor_X \omega_X\geq 0$. Similarly there is a distinguished
  rational differential $\omega_Y$ on $Y$. To say that $\vf$ is volume
  preserving is to say $\omega_X=\omega_Y$ (up to multiplication by a
  nonzero constant) under the identification
  $\Omega^n_{k(X)/k}=\Omega^n_{k(Y)/k}$ given by $\vf$. Equivalently,
  $\vf$ is volume preserving if for all geometric valuations $E$,
  $a(E, K_X+D_X)=a(E, K_Y+D_Y)$. It follows that the notion is
  independent of the choice of common log resolution.

  Volume preserving maps are called crepant birational in \cite{MR3057950}.
\end{rem}

\begin{rem}
  \label{rem:2}
  The composition of two volume preserving maps is volume preserving.
\end{rem}

The first step in the proof of theorem~\ref{thm:1} is the following general
factorisation theorem for volume preserving birational maps between lc
CY pairs, which is of independent
interest. See~\cite[Lemma~12(4)]{kolla15:_calab} for a similar statement.

\begin{thm}
  \label{thm:2}
  Let $(X, D)$ and $(X^\prime, D^\prime)$ be lc CY pairs and $\vf
  \colon X\dasharrow X^\prime$ a volume preserving birational
  map. Then there are $\QQ$-factorial (t,~dlt) CY pairs $(Y, D_Y)$, $(Y^\prime,
  D_{Y^\prime})$ and a commutative diagram of birational maps:
\[
\xymatrix{Y\ar[d]_g\ar@{-->}[r]^\chi & Y^\prime\ar[d]^{g^\prime}\\
X \ar@{-->}[r]^\vf &  X^\prime
}
\]
 where:
 \begin{enumerate}[(1)]
 \item the morphisms $g\colon Y\to X$, $g^\prime\colon Y^\prime \to
   X^\prime$ are volume preserving;
 \item $\chi\colon Y \dasharrow Y^\prime$ is a volume preserving
   isomorphism in codimension $1$ which is a composition of volume
   preserving Mori flips, flops and inverse flips (not necessarily in
   that order).
 \end{enumerate}
 \end{thm}

\begin{dfn}
  \label{dfn:3}
  A \emph{Mori fibred (Mf) CY pair} is a $\QQ$-factorial (t,~lc) CY pair
  $(X,D)$ together with a Mori fibration $f\colon X \to S$. Recall
  that this means that $f_\star \oo_X=\oo_S$, $-K_X$ is $f$-ample, and
  $\rho(X)-\rho(S)=1$.
\end{dfn}

\paragraph{Terminology}
\label{sec:terminology-1}

We use the following terminology throughout.

\begin{itemize}
\item A \emph{Mori divisorial contraction} is an extremal divisorial
  contraction $f\colon Z \to X$ from a $\QQ$-factorial terminal
  variety $Z$ of an extremal ray $R$ with $K_Z\cdot R<0$. In
  particular $X$ also has $\QQ$-factorial terminal singularities. 

  If $(Z, D_Z)$ and $(X,D_X)$ are (t,~lc) CY pairs, then it makes sense to
  say that $f$ is volume preserving. In this context, this is
  equivalent to saying that $K_Z+D_Z=f^\star (K_X+D_X)$ and, in
  particular, $D_X=f_\star (D_Z)$.
\item A birational map $t\colon Z \dasharrow Z^\prime$ is a \emph{Mori
    flip} if $Z$ has $\QQ$-factorial terminal singularities and $t$ is
  the flip of an extremal ray $R$ with $K_Z\cdot R<0$. Note that this
  implies that $Z^\prime$ has $\QQ$-factorial terminal
  singularities. 

  An \emph{inverse Mori flip} is the inverse of a Mori
  flip.  

  A birational map $t\colon Z \dasharrow Z^\prime$ is a
  \emph{Mori flop} if $Z$ and $Z^\prime$ have $\QQ$-factorial terminal
  singularities and $t$ is the flop of an extremal ray $R$ with
  $K_Z\cdot R=0$.

  Again if $(Z,D_Z)$ and $(Z^\prime, D_{Z^\prime})$ are (t,~lc) CY pairs,
  it makes sense to say that $t$ is volume preserving. One can see
  that this just means that $D_{Z^\prime}=t_\star D_Z$.
\end{itemize}

\begin{dfn}
  \label{dfn:4}
  Let $(X, D)$ and $(X^\prime, D^\prime)$ be Mf CY pairs with Mori
  fibrations $X\to S$ and $X^\prime \to S^\prime$. A \emph{volume
    preserving Sarkisov link} is a volume preserving birational map
  $\vf \colon X \dasharrow X^\prime$ that is a Sarkisov link in the
  sense of \cite{MR1311348}. Thus $\vf$ is of one of the following
  types:
  \begin{enumerate}[(I)]
  \item A \emph{link of type I} is a commutative diagram:
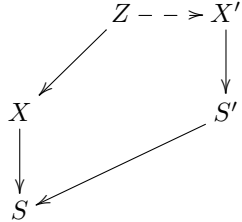
\[
\xymatrix{ &  Z\ar[dl]\ar@{-->}[r] & X^\prime\ar[d] \\
X\ar[d] &  & \ar[lld]S^\prime \\
S & &  
}
\]
where $Z\to X$ is a Mori divisorial contraction and
$Z\dasharrow X^\prime$ a sequence of Mori flips,
flops and inverse flips;
\item A \emph{link of type II} is a commutative diagram:
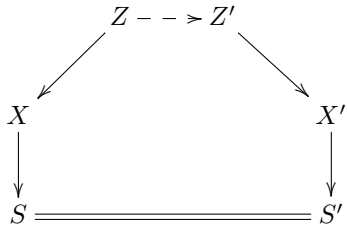
\[
\xymatrix{ &  Z\ar[dl]\ar@{-->}[r] & Z^\prime\ar[dr] & \\
X\ar[d] &  & & X^\prime \ar[d]\\
S\ar@{=}[rrr] & & & S^\prime
}
\]
where $Z\to X$ and $X^\prime \to Z^\prime$ are Mori divisorial
contractions and $Z\dasharrow Z^\prime$ a sequence of Mori flips,
flops and inverse flips;
\item A \emph{link of type III} is the inverse of a link of type I;
\item A \emph{link of type IV} is a commutative diagram:
\[
\xymatrix{X\ar[d]\ar@{-->}[rr] & & X^\prime\ar[d]\\
S\ar[dr] & & S^\prime\ar[dl]\\
&T & & 
}
\]
where $X\dasharrow X^\prime$ is a sequence of Mori flips, flops and
inverse flips.
  \end{enumerate}
\end{dfn}

\begin{rem}
  \label{rem:6}
  It follows from the definition of Sarkisov link that all the
  divisorial contractions, flips, etc.\ that constitute it are
  volume preserving; in particular, all varieties in sight are
  naturally and automatically (t,~lc) CY pairs. 
\end{rem}

In order to appreciate the statement of our main theorem~\ref{thm:1},
it is important to be aware that, although all Mf CY pairs are only
required to have lc singularities \emph{as pairs}, we insist that all
varieties in sight have $\QQ$-factorial \emph{terminal}
singularities. Our factorisation theorem is at the same time a
limiting case of the Sarkisov program for pairs \cite{MR1460896} and a
Sarkisov program for varieties \cite{MR1311348}, \cite{MR3019454}. The
Sarkisov program for pairs usually spoils the singularities of the
underlying varieties, while the Sarkisov program for varieties does
not preserve singularities of pairs. The proof our main result is a
balancing act between singularities of pairs and of varieties.

We expect that it will be possible in some cases to classify
all volume preserving Sarkisov links and hence give useful
presentations of groups of volume preserving birational maps of
interesting Mf CY pairs. We plan to return to these questions in the
near future.

The paper is structured as follows. In Section~\ref{sec:mmp} we
develop some general results on CY pairs and volume preserving maps
between them and prove Theorem~\ref{thm:2}; in Section~\ref{sec:proof}
we prove Theorem~\ref{thm:1}.

\subsection*{Acknowledgements}
\label{sec:acknowledgements}

We thank Paolo Cascini, J\'{a}nos Koll\'{a}r and Vladimir Lazi\'{c}
for useful comments on a preliminary version. 

\section{Birational geometry of CY pairs}
\label{sec:mmp}

\begin{dfn}
  \label{dfn:5}
  Let $(X, D)$ be a lc CY pair, and $f\colon W\to X$ a birational
  morphism. The \emph{log transform} of $D$ is the divisor
\[
D_W=f^\flat (D)=\sum_{a(E, K_X+D)=-1} E 
\]
where the sum is over all prime divisors $E\subset W$.
\end{dfn}

Lemma~\ref{lem:3} is a refinement of~\cite[Theorem 17.10]{MR1225842}
and~\cite[Theorem~4.1]{MR2802603}. In order to state it we need a
definition.

\begin{dfn}
  \label{dfn:10}
 Let $X$ be a normal variety. A \emph{geometric valuation} with centre on $X$
 is a valuation of the function field $K(X)$ of the form $\mult_E$
 where $E\subset Y$ is a divisor on a normal variety $Y$ with a
 birational morphism $f\colon Y\to X$. The \emph{centre} of $E$ on
 $X$, denoted $\centre_X E$, is the generic point of $f(E)$.

 Let $(X,D)$ be a lc pair. The \emph{nonklt set} is the set
\[
\NKLT (X,D)=\{z\in X\mid z=\centre_X E\; \;\text{where}\; a(E,K_X+D)=-1\}
\]
where $E$ is a geometric valuation of the function field of $X$ with
 centre the scheme theoretic point $\centre_X E \in X$.
\end{dfn}

\begin{warn}
  Our notion of nonklt set departs from common usage. Most authors work
  with the \emph{nonklt locus}---the Zariski closure of our
  nonklt set---which they denote $\nklt(X,D)$ (in lower case letters). 
\end{warn}

\begin{rem}
  \label{rem:NKLT}
  We use the following statement throughout. It is part of the
  definition of dlt pairs \cite[Definition~2.37]{MR1658959} that if
  $(X,D)$ is dlt where $D=\sum_{i=1}^r D_i$ with $D_i\subset X$ a
  prime divisor, then $\NKLT(X,D)$ is the set of generic points of the
\[
D_I=\cap_{i\in I} D_i
\quad
\text{where}
\quad
I\subset \{1,\dots, r \}
\]
and $X$ is nonsingular at all these points.
\end{rem}

\begin{lem}
  \label{lem:3}
  Let $(X,D)$ be a lc CY pair where $X$ is not necessarily proper,
  $f\colon W\to X$ a log resolution, and $D_W=f^\flat (D)$. 

The MMP for $K_W+D_W$ over $X$ with scaling of a divisor ample over
$X$ exists and terminates at a minimal model $(Y,D_Y)$ over $X$.

More precisely, this MMP consists of a sequence of steps:
\[
(W,D_W)=(W_0, D_0)\overset{t_0}{\dasharrow}\cdots
(W_i,D_i)\overset{t_i}{\dasharrow} (W_{i+1},D_{i+1})\cdots \dasharrow
(W_N,D_N)=(Y,D_Y) 
\] 
where $t_i\colon W_i\dasharrow W_{i+1}$ is the divisorial contraction
or flip of an extremal ray $R_i\subset \NEbar (W_i/X)$ with
$(K_{W_i}+D_i)\cdot R_i<0$, and we denote by $g_i\colon W_i\to X$ the
structure morphism and by $g\colon (Y, D_Y)\to (X,D)$ the end
result. Then:
  \begin{enumerate}[(1)]
  \item For all $i$, denote by $h_i\colon W\dasharrow W_i$ the induced
    map. For all $i$, there are Zariski open neighbourhoods:
\[
\NKLT (W,D_W)\subset U
\quad\text{and} \quad
\NKLT (W_i,D_i)\subset U_i
\]
such that $h_i|U:U\dasharrow U_i$ is an isomorphism;
  \item $D_Y=g^\flat D$ and $K_Y+D_Y=g^\star (K_X+D)$ (that is, $g$ is
    a dlt crepant blow-up);
  \item $(Y,D_Y)$ is a (t, dlt) CY pair. In particular, $Y$ has terminal singularities;
  \item The map $h\colon W\dasharrow Y$ contracts precisely the prime
    divisors $E\subset W$ with $a(E,K_X+D)>0$. In
    other words,  a $f$-exceptional divisor $E\subset W$ is not
    contracted by the map $h\colon W\dasharrow Y$ if and only if  $a(E,
    K_X+D)=0$ or $-1$.  
  \end{enumerate}
\end{lem}

\begin{proof} The MMP exists by~\cite[Theorem~4.1]{MR2802603}. In the
  rest of the proof we use the following well-known fact: if
  $E$ is a geometric valuation with centre on $W$, then for all $i$
\[
a(E,K_{W_i}+D_i)\leq a(E,K_{W_{i+1}}+D_{i+1})
\]
and the inequality is an equality if and only if $t_i\colon W_i\dasharrow
W_{i+1}$ is an isomorphism in a neighbourhood of $z_i=\centre_{W_i}
E$. In particular, this implies at once that $NKLT(W_i,D_i)\supset
\NKLT(W_{i+1},D_{i+1})$ and the two sets are equal if and only if
there exist Zariski open subsets as in (1), if and only if for all $E$
with $a(E,K+D_i)=-1$ $t_i$ is an isomorphism in a neighbourhood of
$z_{W_i} E$. 

Now write 
\[
K_W + D_W=f^\star (K_X+D)+ F
\]
with $F>0$, $f$-exceptional, with no components in common with
$D_W$. We are running a $F$-MMP, hence if $F_i\subset W_i$ denotes the
image of $F$, then the exceptional set of the map
$t_i\colon W_i \dasharrow W_{i+1}$ is contained in $\Supp F_i$,
see~\cite[\S~1.35]{MR3057950}. From this it follows that $h_i$ is an
isomorphism from $W\setminus \Supp F$ to its image in $W_i$. At the
start, $D_W$ has no components in common with $F$ and
$\Supp(D_W\cup F)$ is a snc divisor: thus, if $a(E, K_W+D_W)=-1$, then
$z_W E\not \in F$. It follows that
$\NKLT(W,D_W)\subset \NKLT(W_i,D_i)$. Together with what we said, this
implies (1).

As for (2), it is obvious that for all $i$ $D_i=g_i^\flat D$. By the
negativity lemma \cite[Lemma~3.39]{MR1658959} $F_i\neq 0$ implies $F_i$ not
nef, so the MMP ends at $g_N=g\colon W_N=Y \to X$ when $F_N=0$, that is,
$K_Y+D_Y=g^\star (K_X+D)$.

For (3) we need to show that $Y$ has terminal singularities. Suppose
that $E$ is a valuation with small centre $\centre_Y E$ on $Y$. By
what we said in Remark~\ref{rem:1}, either $a(E, K_X)>0$, or:
\begin{equation}
  \label{eq:1}
 a(E,K_X)=a(E,K_X+D)=0
\quad
\text{and}
\quad
\centre_Y E \not\in \Supp D_Y
\end{equation}
and we show that this second possibility leads to a
contradiction. Write $z_i=\centre_{W_i} E$. Note that $z_i\in W_i$ is
never a divisor for this would imply that $a(E,K_Y)>0$. By what we
said at the start of the proof, for all $i$, $a(E, K_{W_i}+D_i)\leq
a(E,K_{W_{i+1}}+D_{i+1})$ with strict inequality if and only if
$t_i\colon W_i \dasharrow W_{i+1}$ is not an isomorphism in a
neighbourhood of $z_i\in W_i$. There must be a point where strict
inequality occurs otherwise $z_0\not \in D_0$ and $W=W_0$ is not
terminal in a neibourhood of $z_0$. This, however, implies that $a(E,
K_W+D)<0$, that is, $a(E,K_W+D_W)=-1$ and then by (1) $h\colon W\to Y$
is an isomorphism in a neighbourhood of $z_0$, again a contradiction.

The last statement (4) is obvious. 
\end{proof}

\begin{exa}
  \label{exa:Dflop}
 This example should help appreciate the statement of
 Theorem~\ref{thm:2} and the subtleties of its proof. Let
 $E=\PP^1\times \PP^1$ and $W$ the
 total space of the vector bundle $\oo_E(-1,-2)$. Let $D_W\subset W$
 be a smooth surface such that $D_W\cap E$ is a ruling in $E$ and a
 $-2$-curve in $D_W$. Let $f\colon W \to Y$ be the contraction of $E$
 along the first ruling and $f^\prime \colon W\to Y^\prime$ the
 contraction along the second ruling. Then $(Y,D)$ and $(Y^\prime,
 D^\prime)$ are both dlt, $Y^\prime$ is terminal, $Y$ is canonical
 but not terminal, and the map $Y\dasharrow Y^\prime$ is volume
 preserving. 
\end{exa}

\begin{proof}[Proof of Theorem~\ref{thm:2}]
Let 
\[
\xymatrix{
  & W\ar[dl]_f\ar[dr]^{f'}& \\
X\ar@{-->}[rr]^\varphi &   & X^\prime}
\]
be a common log resolution. Since $\varphi$ is volume preserving, then for
all geometric valuations $E$ $a(E,K_X+D)=a(E,K_{X^\prime}+ D^\prime)$
and:
\[
K_W+D_W= f^\star (K_X+D)+ F= f^{\prime \, \star} (K_{X^\prime}+D^\prime)+ F
\]
where $D_W= f^\flat D = f^{\prime\, \flat} D^\prime$ and 
\[
F= \sum_{a_E(K_X+D)>0}a(E,K_X+D) E = \sum_{a_E(K_{X^\prime}+D^\prime)>0}a(E,K_{X^\prime}+D^\prime) E
\] 

Let $g\colon (Y, D_Y)\to (X, D)$ and
$g^\prime \colon (Y^\prime, D_{Y^\prime})\to (X^\prime,D^\prime)$ be the
end products of the $(K_W+D_W)$-MMP over $X$ and $X^\prime$ as in
Lemma~\ref{lem:3}, and denote by $\chi \colon Y\dasharrow Y^\prime$
the induced map. By Lemma~\ref{lem:3}(4) $\chi$ is an isomorphism in
codimension one.

Denote by $t\colon W\dashrightarrow Y$ and
$t^\prime \colon W\dashrightarrow Y^\prime$ the obvious maps and write
$\NKLT (W,D_W)\subset U_W= W\setminus \Supp F$; by
Lemma~\ref{lem:3}(1) $t{|U_W}$ and $t^\prime {|U_W}$ are
isomorphisms onto their images $\NKLT (Y, D_Y)\subset U\subset Y$ and
$\NKLT (Y^\prime, D_{Y^\prime})\subset U^\prime \subset Y^\prime$. It
follows from this that $\chi|U$ maps $U$ isomorphically to $U^\prime$. 

In the rest of the proof if $N$ is a divisor on $Y$ we denote by
$N^\prime$ its transform on $Y^\prime$ and conversely: because $\chi$ is an
isomorphism in codimension one it is clear what the notation means.
 
Let us choose, as we can by what we just said above, an ample
$\QQ$-divisor $L^\prime$ on $Y^\prime$ general enough that both
$(Y^\prime, D_{Y^\prime}+L^\prime)$ and $(Y, D_{Y}+L)$ are dlt. Let
$0<\varepsilon <\!<1$ be small enough that
$A^\prime = L^\prime -\varepsilon D^\prime$ is ample. Note that, again
by what we said above, writing
$\Theta^\prime=L^\prime +(1-\varepsilon)D_{Y^\prime}$, both pairs
$(Y^\prime,\Theta^\prime)$ and $(Y,\Theta)$ are klt.

Since $K_{Y^\prime}+\Theta^\prime \sim_\QQ A^\prime$ is ample,
$(Y^\prime, \Theta_{Y^\prime})$ is the log canonical model of $(Y,
\Theta)$. It follows that $\chi$ is the composition of finitely
many~\cite[Corollary~1.4.2]{MR2601039} flips
\[ 
\chi \colon Y= Y_0 \overset{\chi_0}{\dashrightarrow} Y_1
\overset{\chi_1}{\dashrightarrow} \cdots \overset{\chi_{N-1}}{\dashrightarrow} Y_N= Y^\prime
\]
of the MMP for $K_Y+\Theta$. If $N$ is a divisor on $Y$, denote by
$N_i$ its transform on $Y_i$. For all $i$, $\chi_i$ is a
$(K_{Y_i}+\Theta_i)$-flip and, at the same time, a
$(K_{Y_i}+D_i)$-flop, and hence all pairs $(Y_i,D_i)$ are lc. We next
argue that all $(Y_i,D_i)$ are in fact (t,dlt).

Because the MMP is a MMP for $A\sim_\QQ K_Y+\Theta$, the exceptional
set of $\chi_i$ is contained in $\Supp A_i$. From this it follows
that, writing $U_0=U$, $\chi_0|U_0$ is an isomorphism onto its image,
which we denote by $U_1$ and, by induction on $i$, $\chi_i|U_i$ is an
isomorphism onto its image, which we denote by $U_{i+1}$. We show by
induction that, for all $i$, $U_i$ is a Zariski neighbourhood of
$\NKLT(Y_i, D_i)$, so that $\chi_i$ is a local isomorphism at the
generic point of each $z\in \NKLT(Y_i, D_i)$ and $(Y_i, D_i)$ is a dlt
pair. Indeed assuming the statement for $i<k$ consider
$\chi_k\colon Y_k \dasharrow Y_{k+1}$. Let $E$ be a valuation with
discrepancy $a(E, K_{Y_{k+1}}+D_{k+1})=-1$, then also
$a(E, K_{Y_k}+D_k)=-1$, thus
$z_k=\centre_{Y_k} E\in \NKLT(Y_k,D_k)\subset U_k$ and then by what we
just said $\chi_k$ is an isomorphism at $z_k$, hence
$z_{k+1}=\chi_k (z_k)\in U_{k+1}$. This shows that all $(Y_i,D_i)$ are
dlt.

Finally we prove that for all $i$ $Y_i$ is terminal. Assume for a
contradiction that $Y_j$ is not terminal. By
Remark~\ref{rem:1}(2) $Y_j$ is canonical and
there is a geometric valuation $E$ with $a(E, K_{Y_j})=a(E,
K_{Y_j}+D_j)=\mult_E \overline{D}_j =0$, and then also $a(E,K_Y+D_Y)=
a(E,K_{Y^\prime}+D_{Y^\prime})=0$. Since $Y$ is terminal, $a(E,
K_Y)>0$, and $\centre_YE\not \in U$, and $\centre_W E \in \Supp F$, but then
$a(E, K_Y+D_Y)>a(E, K_W+D_W)$, so that we must have that $a(E,
K_W+D_W)=-1$, that is $\centre_W E\in \NKLT (W,D_W)\subset U_W$ and
this is a contradiction.
\end{proof}

\section{Sarkisov program under $Y$}
\label{sec:proof}

\subsection{Basic Setup}
\label{sec:basic-setup}

We fix the following situation, which we keep in force throughout this section:
\[
\xymatrix{Y \ar[d]_g \ar@{-->}[rr]^\chi& & Y^\prime \ar[d]^{g^\prime}&  \\
X\ar[d]_p\ar@{-->}[rr]^\varphi & & X^\prime\ar[d]^{p^\prime}\\
S & & S^\prime}
\]
\begin{enumerate}[(i)]
\item $Y$ and $Y^\prime$ have $\QQ$-factorial terminal singularities
  and $g\colon Y \to X$ and $g^\prime \colon Y \to X^\prime$ are birational morphisms.
\item $\chi \colon Y\dasharrow Y^\prime$ is the composition of Mori flips, flops
  and inverse flips. 
\item $p\colon X \to S$ and $p^\prime\colon X^\prime \to S^\prime$ are Mfs.
\end{enumerate}

The goal of this section is to prove theorem~\ref{thm:3} below. In the
final short section~\ref{sec:mainthm} we show that
theorem~\ref{thm:2} and theorem~\ref{thm:3} imply theorem~\ref{thm:1}. The proof of
theorem~\ref{thm:3} is a variation on the proof of
\cite{MR3019454}.

\begin{dfn}
  \label{dfn:6}
  A birational map $f\colon X \dasharrow Y$ is \emph{contracting} if
  $f^{-1}$ contracts no divisors. 
\end{dfn}

\begin{rem}
  \label{rem:8}
  If a birational map $f\colon X \dasharrow Y$ is contracting, then it
  makes sense to pullback $\QQ$-Cartier ($\RR$-Cartier) divisors from
  $Y$ to $X$. Choose a normal variety $W$ and a factorisation:
\[
\xymatrix{ & W\ar[dl]_p \ar[dr]^q& \\
X\ar@{-->}[rr]^f & & Y}
\]
with $p$ and $q$ proper birational morphisms. If $D$ is a
$\QQ$-Cartier ($\RR$-Cartier) $\QQ$-divisor ($\RR$-divisor) on $Y$ the
pullback $f^\star(D)$ is defined as:
\[
f^\star (D)=p_\star q^\star (D)
\]
 (this is easily seen to be independent of the factorisation).
\end{rem}

\begin{thm}
  \label{thm:3}
  The birational map $\varphi\colon X \dasharrow X^\prime$ is a
  composition of links $\varphi_i \colon X_i/S_i \dasharrow
  X_{i+1}/S_{i+1}$ of the Sarkisov program where all the maps
  $Y\dasharrow X_i$ are contracting.
\end{thm}

\begin{term}
  We say that the link $\varphi_i \colon X_i/S_i \dasharrow
  X_{i+1}/S_{i+1}$ is \emph{under $Y$} if the maps $Y\dasharrow X_i$
  and $Y\dasharrow X_{i+1}$ are contracting.
\end{term}

\subsection{Finitely generated divisorial rings}
\label{sec:finit-gener-divis}

\subsubsection{General theory}
\label{sec:general-theory}

\begin{dfn}
  \label{dfn:9}
  Let $f\colon X \dasharrow Y$ be a a contracting birational map.
  
  Let $D_X$ be an $\RR$-divisor. We say that $f$ is
  \emph{$D_X$-nonpositive} (\emph{$D_X$-negative}) if $D_Y=f_\star D_X$ is
  $\RR$-Cartier and
\[
D_X=f^\star(D_Y)+\sum_{\text{E \, \text{$f$-exceptional}}} a_E E
\]
  where all $a_E\geq 0$ (all $a_E>0$). 
\end{dfn}

Note the special case $D_X=K_X$ in the definition just given. 

\begin{dfn}
  \label{dfn:8}
  Let $X/Z$ be a normal variety, proper over $Z$, and $D$ an
  $\RR$-divisor on $X$.
  \begin{enumerate}[(1)]
  \item A \emph{semiample model} of $D$ is a $D$-nonpositive
    contracting birational map $\varphi \colon X \dasharrow Y$ to a
    normal variety $Y/Z$ proper over $Z$ such that $D_Y= \varphi_\star D$
    is semiample over $Z$;
  \item An \emph{ample model} of $D$ is a rational map $h\colon X
    \dasharrow W$ to a normal variety $W/Z$ projective over $Z$
    together with an ample $\RR$-Cartier divisor $A$, such that there
    is a factorisation $h=g \circ f$:
\[
X\overset{f}{\dasharrow} Y \overset{g}{\to} W
\]
where $f\colon X \dasharrow Y$ is a semiample model of $D$, $g\colon Y
\to W$ is a morphism, and $D_Y=g^\star (A)$.
  \end{enumerate}
\end{dfn}

\begin{rem}
  \label{rem:5}
  Let $X/Z$ be a normal variety proper over $Z$ and $D$ an
  $\RR$-divisor on $X$. 
  \begin{enumerate}[(1)]
  \item Suppose that $W/Z$ is normal, $A$ an ample $\RR$-divisor on
    $W$, and $h\colon X\dasharrow W$ an ample model of $D$. If
    $f\colon X \dasharrow Y$ is a semiample model of $D$, then the induced
    rational map $g\colon Y\dasharrow W$ is a morphism and
    $D_Y=g^\star A$.
  \item All ample models of $D$ are isomorphic over $Z$. 
  \end{enumerate}
\end{rem}

We refer to\cite[\S~3]{MoriProcs} for basic terminology on divisorial
rings. 

\begin{thm}\label{thm:decomposition} \cite[Theorem~4.2]{MoriProcs}
  Let $X$ be a projective $\QQ$-factorial variety, and $\mathcal{C}
  \subseteq \Div_\RR(X)$ a rational polyhedral cone containing a big
  divisor\footnote{We need to assume that $\mathcal{C}$ contains a big
    divisor so we can say: if $D\in \mathcal{C}$ is pseudo effective
    then $D$ is effective.} such that the ring $\mathfrak
  R=R(X,\mathcal{C})$ is finitely ge\-ne\-ra\-ted. Then there exists a
  finite rational polyhedral fan $\Sigma$ and a decomposition:
 \[
\Supp\mathfrak R= |\Sigma|= \coprod_{\sigma \in \Sigma} \sigma 
 \]
  such that:
\begin{enumerate}[(1)]
\item For all $\sigma \in \Sigma$ there exists a normal projective
  variety $X_\sigma$ and
  a rational map $\vf_\sigma \colon X\dasharrow X_\sigma$ such that
  for all $D\in\sigma$, $\vf_\sigma$ is the ample model of $D$.
  If $\sigma$ contains a big divisor, then for all
  $D\in\overline{\sigma}$, $\vf_\sigma$ is a
  semiample model of $D$.
\item For all $\tau \subseteq\overline{\sigma}$ there exists
  a morphism $\vf_{\sigma \tau}\colon X_\sigma \longrightarrow X_\tau$
  such that the diagram 
\[
\xymatrix{ 
X \ar@{-->}[rr]^{\vf_\sigma} \ar@{-->}[dr]_{\vf_\tau} & \quad & X_\sigma
  \ar[dl]^{\vf_{\sigma \tau}}\\
\quad & X_\tau & \quad
}\]
commutes.
\end{enumerate}
\end{thm}

\begin{rem}
  \label{rem:4}
  \begin{enumerate}[(1)]
 \item Under the assumptions of Theorem~\ref{thm:decomposition}, if a cone
  $\sigma \in \Sigma$ intersects the interior of
  $\Supp \mathfrak R$, then it consists of big divisors (this is
  because the big cone is the interior of the pseudo-effective cone). This holds in
  particular if $\sigma$ is of maximal dimension.
  \item Theorem~\ref{thm:decomposition}(2) follows
    immediately from part~(1) and remark~\ref{rem:5}(1).
  \end{enumerate}
\end{rem}

\begin{dfn}
  \label{dfn:7}
  Let $X$ be a projective $\QQ$-factorial variety, and
  $\mathcal{C} \subseteq \Div_\RR(X)$ a rational polyhedral cone
  containing a big divisor such that the ring $\mathfrak
  R=R(X,\mathcal{C})$ is finitely ge\-ne\-ra\-ted.  We say that
  $\mathcal{C}$ is \emph{generic} if:
  \begin{enumerate}[(1)]
  \item For all $\sigma \in \Sigma$ of maximal dimension,\footnote{that is,
      $\dim \sigma=\dim \Supp \mathfrak{R}$} $X_\sigma$ is
    $\QQ$-factorial.
  \item For all $\sigma \in \Sigma$, not necessarily of maximal
    dimension, and all $\tau \subset \overline{\sigma}$ of codimension
    one, the morphism $X_\sigma \to X_\tau$ has relative Picard rank
    $\rho({X_\sigma/X_\tau})\leq 1$.
  \end{enumerate}
\end{dfn}

\begin{nota}
  If $V$ is a $\RR$-vector space and $v_1, \dots, v_k \in V$, then we
  denote by
\[
\langle v_1,\dots, v_k \rangle = \sum_{i=1}^k \RR_{\geq 0} v_i
\]
 the convex cone in $V$ spanned by the $v_i$. 
\end{nota}

\begin{lem}
  \label{lem:2}
  Let $X$ be a projective $\QQ$-factorial variety, and
  $\mathcal{C} \subseteq \Div_\RR(X)$ a generic rational polyhedral cone
  containing a big divisor.

  Let $D_1, \dots, D_k\in \mathcal{C}$ such that the cone
  $\langle D_1, \dots, D_k\rangle$ contains a big divisor, and let
  $\ve >0$. There exist $D_1^\prime, \dots D^\prime_k \in \mathcal{C}$
  with $|\!|D_i-D_i^\prime|\!|<\ve$ such that the cone
  $\langle D_1^\prime, \dots D^\prime_k \rangle$ is generic.
\end{lem}

\begin{proof}
  Make sure that all cones $\langle D^\prime_{i_1}, \dots,
  D^\prime_{i_c}\rangle$, $i_1, \dots, i_c \in \{1, \dots, k\}$,
  intersect all cones $\sigma \in \Sigma$ properly.
\end{proof}

\begin{thm}
  \label{pro:1}
  Let $X$ be a projective $\QQ$-factorial variety, $\Delta_1, \dots,
  \Delta_r\geq 0$ big $\QQ$-divisors on $X$ such that all pairs
  $(X,\Delta_i)$ are klt, and let 
\[
\mathcal{C} =\langle K_X+\Delta_1, \dots, K_X+\Delta_r \rangle
\] 

Then $\mathfrak{R}=R(X,\mathcal{C})$ is finitely ge\-ne\-ra\-ted, and
if $\Supp \mathfrak{R}$ spans $N^1_\RR(X)$ as a vector space then
$\mathcal{C}$ is generic. \qed
\end{thm}

For the proof see for example \cite[Theorem~4.5]{MoriProcs}. Note that
the assumptions readily imply that $\Supp \mathfrak{R}$ contains big
divisors. The finite generation of $\mathfrak{R}$ is of course the big
theorem of \cite{MR2601039}.

\begin{setup}
  \label{setup:1}
In what follows we work with a pair $(X, G_X)$ where $X$ is
$\QQ$-factorial and:
\begin{enumerate}[(i)]
\item $G_X$ is a $\QQ$-linear combination of irreducible
  mobile\footnote{A $\QQ$-divisor $M$ is mobile if for some integer
    $n>0$ such that $nM$ is integral, the linear system $|nM|$ has no
    fixed (divisorial) part.} divisors;
\item $(X,G_X)$ is terminal;
\item $K_X+G_X$ is not pseudoeffective.
\end{enumerate}
\end{setup}
Assumption~(i) implies that when running the MMP for $\Kc{X}{G}$ no
component of $G_X$ is ever contracted, so that $(X,G_X)$ remains terminal
throughout the MMP. Assumption~(iii) means that the MMP terminates
with a Mf.

\begin{cor}
  \label{cor:1}
  Let $X$ be a projective $\QQ$-factorial variety, $G_X$ as in
  setup~\ref{setup:1}, $\Delta_1, \dots, \Delta_r\geq 0$ big
  $\QQ$-divisors on $X$ such that all pairs $(X,G_X+\Delta_i)$ are
  klt.

 Then for all $\varepsilon >0$ there are ample $\QQ$-divisors $H_1, \dots,
  H_r\geq 0$ with $|\!|H_i|\!|<\varepsilon$ such that
\[
\mathcal{C}^\prime =\langle \Kc{X}{G}, \Kc{X}{G}+\Delta_1+H_1, \dots,
\Kc{X}{G}+\Delta_r+H_r \rangle
\] 
is generic. 
\end{cor}

\begin{proof}
  Add enough ample divisors to span $N^1$ and then use
  lemma~\ref{lem:2} to perturb $\Delta_1, \dots, \Delta_r$ inside a
  bigger cone. Since $\Kc{X}{G}\not \in \Effbar X$, then
  $\Kc{X}{G}\not \in \Supp \mathfrak{R} (X, \mathcal{C})$, and hence
  there is no need to perturb $G_X$. 
\end{proof}

\subsubsection{Special case: $2$-dimensional cones}
\label{sec:special-case:-2}

Suppose that $A$ is a big $\QQ$-divisor on $X$ such that
\begin{enumerate}[(i)]
\item the pair $(X,G_X+ A)$ is klt;
\item $\Kc{X}{G}+A$ is ample on $X$;
\item the cone $\mathcal{C}=\langle \Kc{X}{G}, \Kc{X}{G}+A\rangle$ is
  generic.
\end{enumerate}

Then, the decomposition of $\Supp \mathfrak{R}(X, \mathcal{C})$ given
by theorem~\ref{thm:decomposition} corresponds to running a MMP for
$\Kc{X}{G}$ with scaling by $A$. This MMP exist by
\cite[Corollary~1.4.2]{MR2601039}. In more detail, let
\[1=t_0>t_1>\cdots>t_{N+1}>0
\]
be rational numbers such that $\Supp
\mathfrak{R}(X,\mathcal{C})=\langle \Kc{X}{G}+A,
\Kc{X}{G}+t_{N+1}A\rangle$ and the maximal cones of the decomposition
correspond to the intervals $(t_i, t_{i+1})$. For all $t\in
(t_i,t_{i+1})$ $\Kc{X}{G}+tA$ is ample on $X_i=\Proj R(X,
\Kc{X}{G}+tA)$. Then
\[X=X_0\dasharrow X_1 \dasharrow \cdots \dasharrow X_i\dasharrow
X_{i+1}\dasharrow \cdots \dasharrow X_N\]
is a minimal model program for $\Kc{X}{G}$ with scaling by $A$, that is: 
\begin{enumerate}[(1)]
\item \[t_{i+1}=\inf \{\tau \in \RR \mid \Kc{X_i}{G} +\tau A_i \;
  \text{is nef} \}\] where $A_i$ denotes the push forward of $A$, and
  $X_i\dasharrow X_{i+1}$ is the divisorial contraction or flip of an
  extremal ray $R_i\subset \NEbar (X_i)$ with
  \[(\Kc{X_i}{G} +t_{i+1}A_i)\cdot R_i=0\quad
\text{and} \quad (\Kc{X_i}{G})\cdot R_i<0\]
and
\item $t_{N+1}=\inf\{\lambda \mid \Kc{X}{G}+\lambda A\; \text{is effective}
  \}$\footnote{i.e., $\Kc{X}{G}+\lambda A$ is $\QQ$-linearly equivalent
    to an effective divisor.} and 
\[X_N\to\Proj R(X, \Kc{X}{G}+t_{N+1}A)\]
 is a Mf. Moreover:
\item Genericity means that at each step there is a
unique extremal ray $R_i\subset \NEbar (X_i)$ with
$(\Kc{X_i}{G}+t_{i+1}A_i)\cdot R_i=0$ and $(\Kc{X_i}{G})\cdot R_i<0$.
\item The following follows immediately from genericity. If
  $0<\varepsilon<\!<1$ is small enough then for all ample
  $\QQ$-divisors $H$ with $|\!|H|\!|<\varepsilon$, $\Kc{X}{G}+A+H$ is
  ample on $X$, the cone $\mathcal{C}^\prime=\langle \Kc{X}{G},
  \Kc{X}{G}+A^\prime\rangle$ is still generic, and the MMP for
  $\Kc{X}{G}$ with scaling by $A+H$ is identical to the MMP for
  $\Kc{X}{G}$ with scaling by $A$, in the sense that the sequence of
  steps and end product are identical.
\end{enumerate}

\subsubsection{Special case: $3$-dimensional cones}
\label{sec:special-case:-3}

In this subsection we prove the following special case of
theorem~\ref{thm:3}:

\begin{lem}
  \label{lem:6}
  Suppose that $(Y,G_Y)$ is as in setup~\ref{setup:1}, and that
  $A$, $A^\prime$ are big $\QQ$-divisors on $Y$ such that:
  \begin{enumerate}[(i)]
  \item $(Y, G_Y+A)$ and $(Y, G_Y+A^\prime)$ are klt;
  \item $\Kc{Y}{G}+A$ and $\Kc{Y}{G}+A^\prime$ are both ample on $Y$;
  \item the cones $\mathcal{C}=\langle \Kc{Y}{G}, \Kc{Y}{G}+A
    \rangle$ and $\mathcal{C}^\prime=\langle \Kc{Y}{G}, \Kc{Y}{G}+A^\prime
    \rangle$ are generic;
  \item the MMP for $\Kc{Y}{G}$ with scaling by $A$, resp.\ $A^\prime$,
    ends in a Mf $X/S$, resp.\ $X^\prime/S^\prime$.
  \end{enumerate}
  Then the birational map $\varphi\colon X \dasharrow X^\prime$ is a
  composition of links $\varphi_i \colon X_i/S_i \dasharrow
  X_{i+1}/S_{i+1}$ of the Sarkisov program where each map
  $Y\dasharrow X_i$ is contracting.
\end{lem}

\begin{proof}
  The proof is the argument of \cite{MR3019454}, which we sketch here
  for the reader's convenience. After a small perturbation of 
  $A$ and $A^\prime$ as in corollary~\ref{cor:1} that, as stated in
  \S~\ref{sec:special-case:-2}(4), does not change the two MMPs or
  their end products, the cone
  $\widetilde{\mathcal{C}}= \langle \Kc{Y}{G}, \Kc{Y}{G}+A,
  \Kc{Y}{G}+A^\prime \rangle$
  is generic. The argument of \cite{MR3019454} then shows how walking
  along the boundary of $\Supp \widetilde{\mathcal{C}}$ corresponds
  to a chain of Sarkisov links from $X/S$ to $X^\prime/S^\prime$. By
  construction, all maps from $Y$ are contracting.
\end{proof}

\subsection{Proof of Theorem~\ref{thm:3}}
\label{sec:proof-theorem-1}

 Write $\chi=\chi_{N-1} \circ \cdots \circ \chi_0$ where each 
\[
\chi_i\colon Y_i \dasharrow Y_{i+1}
\]
is a Mori flip, flop or inverse flip, and $Y=Y_0$, $Y^\prime=Y_{N}$. 

For all $\QQ$-divisors $G_Y$ on $Y$ denote by $G_{Y_i}$ the strict
transform on $Y_i$. Choose $G_Y$ such that for all $i\in \{0,\dots,N\}$:
\begin{enumerate}[(i)]
\item $G_{Y_i}$ satisfies the setup~\ref{setup:1}, and
\item $\chi_i$ is either a $(\Kc{Y_i}{G})$-flip or
  antiflip. \footnote{The purpose of $G$ is to make sure that there are
  no flops.}
\end{enumerate}
One way to choose $G_Y$ is as follows: if $\chi_i$ is a flop, choose
$G_{Y_i}$ ample, general, and very small on $Y_i$. If $G_{Y_i}$ is
small enough then for all $j$ if $\chi_j$ was a flip or antiflip then
it still is a flip or antiflip. On the other hand now $\chi_i$ is a
$(\Kc{Y_i}{G})$-flip. Some other flops may have become flips or
antilfips. If there are still flops repeat the process by adding, on
$Y_k$ such that $\chi_k$ is a $(\Kc{Y_k}{G})$-flop, a very small ample
divisor to $G_{Y_k}$, and so on until there are no flops left.

For all $i\in \{0,\dots, N\}$, we choose by induction on $i$ a big
divisor $A_i$ on $Y_i$\footnote{Here $A_j$ is not the transform of
  $A_i$ on $Y_j$: it is just another divisor.} such that
$\Kc{Y_i}{G}+A_i$ is ample, $\langle \Kc{Y_i}{G},
\Kc{Y_i}{G}+A_i\rangle$ is generic, and the MMP for $\Kc{Y_i}{G}$ with
scaling by $A_i$ terminates with a Mf $p_i\colon X_i \to S_i$. At the
start $p_0=p \colon X_0=X\to S_0=S$, but it will not necessarily be
the case that $p_N=p^\prime$. We prove, also by induction on $i$, that
for all $i$ the induced map $\varphi_i\colon X_i\dasharrow X_{i+1}$ is the
composition of Sarkisov links under $Y$. Finally we prove that the
induced map $X_N \dasharrow X^\prime$ is the composition of Sarkisov
links under $Y$.

Suppose that for all $j<i$ $A_j$ has been constructed. We consider two cases:
\begin{enumerate}[(a)]
\item If $\chi_{i-1}$ is a $(\Kc{Y_{i-1}}{G})$-flip, choose an ample
  $\QQ$-divisor $A_{i-1}^\prime$ on $Y_{i-1}$ such that
  $\langle \Kc{Y_{i-1}}{G}, \Kc{Y_{i-1}}{G}+A_{i-1}^\prime \rangle$ is
  generic and the MMP for $\Kc{Y_{i-1}}{G}$ with scaling by
  $A_{i-1}^\prime$ begins with the flip $\chi_{i-1}$. This can be
  accomplished as follows: if $\chi_{i-1}$ is the flip of the extremal
  contraction $\gamma_{i-1}\colon Y_{i-1}\to Z_{i-1}$ then
  $A^\prime_{i-1}=L_{i-1}+\gamma_{i-1}^\star (N_{i-1})$ where
  $L_{i-1}$ is ample on $Y_{i-1}$ and $N_{i-1}$ is ample enough on
  $Z_{i-1}$. Now set 
\[A_i=\chi_{i-1\,\star } \Bigl((t_1-\varepsilon) A_{i-1}^\prime\Bigr)
\] where $\Kc{Y_{i-1}}{G}+t_1A^\prime_{i-1}$ is $\gamma_{i-1}$-trivial
and $0<\varepsilon <\!<1$. Note that $\langle \Kc{Y_{i-1}}{G},
\Kc{Y_{i-1}}{G}+A_{i-1}^\prime \rangle$ generic implies $\langle
\Kc{Y_i}{G}, \Kc{Y_i}{G}+A_i \rangle$
generic.\footnote{$A_{i-1}^\prime$ is ample hence (moving in linear
  equivalence class) $(Y_{i-1}, G_{Y_{i-1}}+A^\prime_{i-1})$ is
  klt--in fact even terminal if we want. So since $t_1<1$ then even
  more $(Y_i, G_{Y_i}+A_i)$ is klt.} We take $p_i\colon X_i\to S_i$ to
be the end product of the MMP for $\Kc{Y_i}{G}$ with scaling by
$A_i$. It follows from lemma~\ref{lem:6}, applied to $Y_{i-1}$ and
the divisors $A_{i-1}$, $A_{i-1}^\prime$, that the induced map
$\varphi_i\colon X_{i-1}\dasharrow X_i$ is a composition of Sarkisov
links under $Y_{i-1}$ and hence, since $Y\dasharrow Y_{i-1}$ is an
isomorphism in codimension one, under $Y$.
\item If $\chi_{i-1}$ is a $(\Kc{Y_{i-1}}{G})$-antiflip, choose $A_i$ ample
  on $Y_i$ such that
  $\langle \Kc{Y_i}{G}, \Kc{Y_i}{G}+A_i \rangle$ is
  generic and the MMP for $\Kc{Y_i}{G}$ with scaling
  by $A_i$ begins with the flip $\chi_{i-1}^{-1}$. We take
  $p_i\colon X_i\to S_i$ to be the end product of the MMP for
  $\Kc{Y_i}{G}$ with scaling by $A_i$. It follows from
  lemma~\ref{lem:6}, applied to $Y_{i-1}$ and the divisors $A_{i-1}$,
  \[
A_{i-1}^\prime =\chi_{i-1\, \star}^{-1} \Bigl((t_1-\varepsilon) A_i\Bigr)
\] where $\Kc{Y_i}{G}+t_1A_i$ is $\chi_{i-1}^{-1}$-trivial and
$0<\varepsilon <\!<1$, that the induced map $\varphi_{i-1}\colon
X_{i-1}\dasharrow X_i$ is a composition of Sarkisov links under
$Y_{i-1}$ and hence, since $Y\dasharrow Y_{i-1}$ is an isomorphism in
codimension one, under $Y$.
\end{enumerate}
Finally, choose $A^\prime$ ample on $Y^\prime$ such that $\langle
\Kc{Y^\prime}{G}, \Kc{Y^\prime}{G}+A^\prime\rangle$ is generic and the
MMP for $\Kc{Y^\prime}{G}$ with scaling by $A^\prime$ terminates with
the Mf $p^\prime \colon X^\prime \to S^\prime$. It follows from
lemma~\ref{lem:6}, applied to $Y_N=Y^\prime$ (and the divisors $A_N$,
$A^\prime$), that the induced map $\varphi_N\colon X_r\dasharrow
X^\prime$ is a composition of Sarkisov links under $Y^\prime$ and
hence, since $Y\dasharrow Y^\prime$ is an isomorphism in codimension
one, under $Y$. \qed

\section{Proof of theorem~\ref{thm:1}}
\label{sec:mainthm}

Let $(X,D)$ and $(X^\prime, D^\prime)$ with $p\colon X \to S$ and
$p^\prime \colon X^\prime \to S^\prime$ be (t,lc) Mf CY pairs, and let
$\varphi\colon X\dasharrow X^\prime$ be a volume preserving birational
map. Theorem~\ref{thm:2} gives a diagram:
\[
\xymatrix{(Y,D_Y)\ar[d]_g\ar@{-->}[r]^\chi & (Y^\prime,D_{Y^\prime})\ar[d]^{g^\prime}\\
(X,D) \ar@{-->}[r]^\vf &  (X^\prime,D^\prime)
}
\]
where $(Y, D_Y)$, $(Y^\prime, D_{Y^\prime})$ are (t,dlt)
$\QQ$-factorial CY pairs, $g$, $g^\prime$ are volume preserving and
$\chi$ is a volume preserving composition of Mori flips, flops and
inverse flips. In particular, if we forget the $D$s, we are in the
situation of \S~\ref{sec:basic-setup}, so that by theorem~\ref{thm:3}
$\varphi \colon X \dasharrow X^\prime$ is the composition of Sarkisov
links $\varphi_i\colon X_i/S_i\dasharrow X_{i+1}/S_{i+1}$ such that
all induced maps $g_i\colon Y\dasharrow X_i$ are contracting. Denoting
by $D_i=g_{i\, \star} D_Y$, it is clear that for all $i$ $g_i$ is
volume preserving hence also
$\varphi_i\colon (X_i, D_i)\dasharrow (X_{i+1},D_{i+1})$ is volume
preserving, and $(X_i,D_i)$ is a (t,lc) CY pair. \qed




\begin{thebibliography}{10}

\bibitem{MR2601039}
Caucher Birkar, Paolo Cascini, Christopher~D. Hacon, and James McKernan.
\newblock Existence of minimal models for varieties of log general type.
\newblock {\em J. Amer. Math. Soc.}, 23(2):405--468, 2010.

\bibitem{MR3080816}
J{\'e}r{\'e}my Blanc.
\newblock Symplectic birational transformations of the plane.
\newblock {\em Osaka J. Math.}, 50(2):573--590, 2013.

\bibitem{MR1460896}
Andrea Bruno and Kenji Matsuki.
\newblock Log {S}arkisov program.
\newblock {\em Internat. J. Math.}, 8(4):451--494, 1997.

\bibitem{MR1311348}
Alessio Corti.
\newblock Factoring birational maps of threefolds after {S}arkisov.
\newblock {\em J. Algebraic Geom.}, 4(2):223--254, 1995.

\bibitem{MR2802603}
Osamu Fujino.
\newblock Semi-stable minimal model program for varieties with trivial
  canonical divisor.
\newblock {\em Proc. Japan Acad. Ser. A Math. Sci.}, 87(3):25--30, 2011.

\bibitem{MR3019454}
Christopher~D. Hacon and James McKernan.
\newblock The {S}arkisov program.
\newblock {\em J. Algebraic Geom.}, 22(2):389--405, 2013.

\bibitem{MoriProcs}
Anne-Sophie Kaloghiros, Alex K\"uronya, and Vladimir Lazi\'c.
\newblock Finite generation and geography of models.
\newblock In Shigeru Mukai, editor, {\em Minimal {M}odels and {E}xtremal
  {R}ays}, {A}dvanced {S}tudies in {P}ure {M}athematics. {M}athematical
  {S}ociety of {J}apan, to appear.

\bibitem{MR3057950}
J{\'a}nos Koll{\'a}r.
\newblock {\em Singularities of the minimal model program}, volume 200 of {\em
  Cambridge Tracts in Mathematics}.
\newblock Cambridge University Press, Cambridge, 2013.
\newblock With a collaboration of S{\'a}ndor Kov{\'a}cs.

\bibitem{MR1225842}
J\'{a}nos Koll\'{a}r et~al.
\newblock {\em Flips and abundance for algebraic threefolds}.
\newblock Soci\'et\'e Math\'ematique de France, Paris, 1992.
\newblock Papers from the Second Summer Seminar on Algebraic Geometry held at
  the University of Utah, Salt Lake City, Utah, August 1991, Ast{\'e}risque No.
  211 (1992).

\bibitem{MR1658959}
J{\'a}nos Koll{\'a}r and Shigefumi Mori.
\newblock {\em Birational geometry of algebraic varieties}, volume 134 of {\em
  Cambridge Tracts in Mathematics}.
\newblock Cambridge University Press, Cambridge, 1998.
\newblock With the collaboration of C. Herbert Clemens and Alessio Corti,
  Translated from the 1998 Japanese original.

\bibitem{kolla15:_calab}
J\'{a}nos Koll\'{a}r and Chenyang Xu.
\newblock The dual complex of {C}alabi--{Y}au pairs.
\newblock
  \href{{http://arxiv.org/abs/1503.08320}}{\texttt{arXiv:math/1503.08320
  [math.AG]}}, 2015.

\bibitem{usnich06:_sympl_cp_thomp_t}
Alexandr Usnich.
\newblock Symplectic automorphisms of $\mathbb{CP}^2$ and the {T}hompson group
  {T}.
\newblock \href{http://arxiv.org/abs/math/0611604}{\texttt{arXiv:math/0611604
  [math.AG]}}, 2006.

\end{thebibliography}

   \end{document}